\documentclass[11pt,reqno]{amsart}
\usepackage[top=1.2in, bottom=1.2in, left=1.2in, right=1.2in]{geometry}
\usepackage{amsfonts, amssymb, amsmath, mathtools, dsfont}
\usepackage[linktoc=page, colorlinks, linkcolor=blue, citecolor=blue]{hyperref}
\usepackage{enumerate}

\usepackage{graphicx}

\numberwithin{equation}{section}

\DeclareMathOperator{\E}{\mathbb{E}}

\DeclareMathOperator{\supp}{supp}

\DeclareMathOperator{\sgn}{sgn}
\DeclareMathOperator*{\argmax}{argmax}

\DeclarePairedDelimiter \abs{\lvert}{\rvert} 
\DeclarePairedDelimiter \norm{\lVert}{\rVert}
\DeclarePairedDelimiterX \ip[2]{\langle}{\rangle}{#1,#2}
\DeclarePairedDelimiterXPP \Prob[1]{\mathbb{P}}\{\}{}{
   
   #1} 
\DeclarePairedDelimiterXPP \Probevent[1]{\mathbb{P}}(){}{
   
   #1} 

\def \R {\mathbb{R}}

\newtheorem{theorem}{Theorem}[section]

\newtheorem{lemma}[theorem]{Lemma}

\newtheorem{definition}[theorem]{Definition}

\theoremstyle{remark}
\newtheorem{remark}[theorem]{Remark}

\newtheorem{question}[theorem]{Question}

\begin{document}

\title{Can we spot a fake?}

\author{Shahar Mendelson}
\address{Texas A\&M University}
\email{shahar@tamu.edu}
\author{Grigoris Paouris}
\address{Texas A\&M University and Princeton University}
\email{grigoris@tamu.edu}
\author{Roman Vershynin}
\address{University of California, Irvine}
\email{rvershyn@uci.edu}

\begin{abstract}
    The problem of detecting fake data inspires the following seemingly simple mathematical question. Sample a data point $X$ from the standard normal distribution in $\mathbb{R}^n$. An adversary observes $X$ and corrupts it by adding a vector $rt$, where they can choose any vector $t$ from a fixed set $T$ of the adversary's ``tricks'', and where $r>0$ is a fixed radius. The adversary's choice of $t=t(X)$ may depend on the true data $X$. The adversary wants to hide the corruption by making the fake data $X+rt$ statistically indistinguishable from the real data $X$. What is the largest radius $r=r(T)$ for which the adversary can create an undetectable fake? We show that for highly symmetric sets $T$, the detectability radius $r(T)$ is approximately twice the scaled Gaussian width of $T$. The upper bound actually holds for arbitrary sets $T$ and generalizes to arbitrary, non-Gaussian distributions of real data $X$. The lower bound may fail for not highly symmetric $T$, but we conjecture that this problem can be solved by considering the focused version of the Gaussian width of $T$, which focuses on the most important directions of $T$.
\end{abstract}

\maketitle

\section{Introduction}

The ongoing AI boom is revolutionizing human-computer interaction. But not all interaction is good: generative AI has made it easy to create fake data. 

\begin{quote}
{\em When can we detect a fake?} 
\end{quote}

This general question can be nontrivial even in stunningly simple scenarios. Imagine we sample a data point $X$ from the standard normal distribution on $\R^n$. An adversary observes $X$ and corrupts it by adding a vector $rt$, where they can choose any vector $t=t(X) \in T$ from a fixed set $T \subset \R^n$ of adversary's ``tricks'', and where $r>0$ is a fixed radius. Allowing the adversary's choice of $t=t(X)$ to depend on $X$ allows the adversary to manipulate different data in different ways. The adversary wants to hide the corruption by making the fake data $X+rt$ statistically indistinguishable from the real data $X$. When can the adversary succeed?

\begin{question}[Detectability radius]	\label{q: detectability radius}
    {\em For a given set $T \subset \R^n$, what is the smallest radius $r(T)$ such that for any $r>r(T)$ we can always detect a fake with high probability?}
\end{question}

If $r$ is very small, the fake can never be detected, since the adversary can always succeed: a tiny shift of the normal distribution is close to the original normal distribution. If $r$ is very large, the fake can always be detected: a hugely shifted vector has an abnormally large length, and this can be easily detected. So where is the phase transition $r(T)$? And what are the optimal strategies for the tester and for the adversary?

\subsection{A geometric viewpoint}
Any test can be encoded by a subset $A \subset \R^n$: for any point in $A$ the tester says ``true'', and for any point in the complement of $A$ the tester says ``fake''. A good test recognizes true data as true, and does {\em not} recognize fake data as true, both with a high probability. Thus, the tester's goal is to find a set $A$ that minimizes the probability of error, so that
\begin{equation}    \label{eq: adversarial problem}
\Prob{X \in A} \ge 0.9
\quad \text{and} \quad
\Prob{\exists t \in T:\; X+rt \in A} \le 0.1.
\end{equation}

We have chosen the values $0.9$ and $0.1$ for illustrative purposes only; they can be arbitrary in general.

Using Minkowski addition, the second inequality in \eqref{eq: adversarial problem} can be expressed as $\Prob{X \in A-rT} \le 0.1$. Thus, denoting by $\gamma_n$ the standard Gaussian measure on $\R^n$, the tester's goal is to find a subset $A \subset \R^n$ satistying
\begin{equation}	\label{eq: A}
	\gamma_n(A) \ge 0.9
	\quad \text{and} \quad
	\gamma_n(A-rT) \le 0.1.
\end{equation}
The {\em detectablity radius} we introduced in Question~\ref{q: detectability radius} can be formally defined as the largest radius of an undetectable fake:
$$
r(T) \coloneqq \sup \left\{ r>0:\; \text{there exists no set $A \subset \R^n$ that satisfies \eqref{eq: A}} \right\}.
$$

\subsection{Related work}
A lot of effort has been made to compute the detectability radius for an {\em outsider} adversary whose choice of the vector $t \in T$ is {\em not} allowed to depend on $X$. Instead of \eqref{eq: A}, in the ``outsider adversary'' scenario the tester is satisfied with the following weaker goal: find a subset $A \subset \R^n$ that satisfies
\begin{equation}	\label{eq: A sup outside}
	\gamma_n(A) \ge 0.9
	\quad \text{and} \quad
	\sup_{t \in T} \gamma_n(A-rt) \le 0.1.
\end{equation}
The study of this ``outsider adversary'' problem can be traced back to Tuckey's idea of {\em ``higher criticism''} \cite{tuckey}, see \cite{DJ1, DJ2, DJ3}. Higher criticism was recently revisited from the standpoint of high-dimensional statistics, and was studied for various specific sets $T$ including the spheres in the $\ell^p$ metric \cite{Ingster}, the indicators of paths and more general clusters in a given graph \cite{ACHZ, ACD}, the set of unit sparse vectors \cite{Baraud, Ingster2, DJ1, DJ2, DJ3, CJL, HJ, IPT}, and more generally a set containing all unit vectors with given sparsity patterns \cite{ABDL}. Extensions for sparse regression has been studied in \cite{ACP, ITV, CCCTW, MS}. The only work on the ``insider adversary'' that we know of is the concurrent work \cite{Smirnov}, which studies the detectability radius of the discrete cube $T=\{-1,1\}^n$ and its sparsified versions.

\subsection{Main results}

Returning to our ``insider adversary'' problem, it turns out that in many situations,  the detectability radius $r(T)$ is captured by the quantity that we call the {\em scaled Gaussian width} of $T$. It is defined as follows:
\begin{equation}		\label{eq: scaled Gaussian width}
\bar{w}(T) = \E \sup_{t \in T} \ip*{X}{\frac{t}{\norm{t}_2^2}}
\end{equation}
where $X$ is a standard normal random vector in $\R^n$. This is a scaled version of the more traditional quantity called {\em Gaussian width}, which is 
\begin{equation}		\label{eq: Gaussian width}
w(T) = \E \sup_{t \in T} \ip{X}{t}.
\end{equation}
The Gaussian width describes the complexity of the set $T$. This concept plays an important role in high-dimensional probability, asymptotic convex geometry and high-dimensional inference \cite{V book, AGM1, AGM2}. 

The unorthodox scaling in \eqref{eq: scaled Gaussian width} is justified by the contribution of the points in $T$ near the origin. Since smaller distortions are easier to hide, having such points in $T$ makes the adversary's job easier and the tester's job harder. So a quantity that captures the detectability radius $r(T)$ should be more sensitive to points in $T$ that are close to the origin than to those that are far from the origin. Such a sensitivity is achieved in \eqref{eq: scaled Gaussian width} by dividing by $\norm{t}_2^2$.

The following theorem summarizes our main results. 

\begin{theorem}[Detectability radius, informal]	\label{thm: detectability radius, informal}
	For any set $T \subset \R^n$, we have 
	\begin{equation} \label{eq: critical radius upper}
	r(T) \le 2\bar{w}(T) \left(1+o(1)\right).
	\end{equation}
	Moreover, for any highly symmetric set $T$ we have 
	\begin{equation} \label{eq: critical radius lower}
	r(T) \ge 2\bar{w}(T) \left(1-o(1)\right).
	\end{equation}
\end{theorem}

The first part of Theorem~\ref{thm: detectability radius, informal} identifies the regime in which the fake is detectable. A formal, non-asymptotic version of \eqref{eq: critical radius upper} is provided in Theorem~\ref{thm: upper}. We argue in Section~\ref{s: proximity test} that a simple {\em proximity test} can detect a fake in this regime: if the point is closer to the origin than to any point in $T$, return ``real''; otherwise return ``fake''.
   
The second part of Theorem~\ref{thm: detectability radius, informal} identifies the regime in which there exist undetectable fakes. The meaning of ``highly symmetric'' is explained in  Definition~\ref{def: highly symmetric set}, and a formal, non-asymptotic version of \eqref{eq: critical radius upper} is provided in Theorem~\ref{thm: lower}. We argue in Section~\ref{s: sign flipping} that an undetectable fake can be produced in this regime by a {\em sign flipping strategy}: choose a point $t(X) \in T$ such that adding $rt$ to the outcome $X$ is equivalent to reversing the signs of some coordinates of $X$. 

\subsection{Example: the adversary adds any vector}

To illustrate the intuition behind Theorem~\ref{thm: detectability radius, informal}, consider the most basic example where the adversary is allowed to add to $X$ any vector of length at least $r$. In other words, the set of the adversary's tricks is 
\begin{equation}	\label{eq: Tn}
	T = T_n = \left\{ t \in \R^n:\; \norm{t}_2 \ge 1 \right\}.
\end{equation}
The scaled Gaussian width of $T$ equals the Gaussian width of the unit Euclidean ball in $\R^n$, and it is approximately $\sqrt{n}$, see \cite[Example~7.5.7]{V book}. Thus Theorem~\ref{thm: detectability radius, informal} gives
$$
r(T_n) = 2\sqrt{n}(1+o(1)).
$$

There is a simpler way to obtain the same conclusion. The Euclidean norm of the standard normal random vector $X$ is approximately $\sqrt{n}$ \cite[Example~7.5.3]{V book}. Therefore, if the radius $r$ is significantly larger than $2\sqrt{n}$, adding to $X$ a vector $rt$ for any $t \in T$ would make the norm of the sum significantly larger than $\sqrt{n}$. So the norm of the fake data $X+rt$ must be significantly larger than the typical norm of the real data $X$. This abnormality can be easily detected.

On the other hand, let us take $r=2\sqrt{n}$ and simplify reality a little by pretending that the Euclidean norm of $X$ is exactly $\sqrt{n}$. Then the adversary can make an undetectable fake by adding the vector $-2X \in rT$ to $X$. Indeed, this effectively flips the vector $X$ about the origin, which does not change the distribution of $X$ at all. 

\subsection{Example: the adversary adds a sparse vector}    \label{s: sparsity}

To make the previous example more interesting, imagine that the adversary has limited resources and can change at most $s$ coordinates of $X$. The choice of which coordinates to change is up to the adversary and may depend on $X$. In other words, the adversary's set of tricks is 
\begin{equation}	\label{eq: Tns}
	T = T_{n,s} = \left\{ t \in \R^n:\; \norm{t}_2 \ge 1, \; \norm{t}_0 \le s \right\}.
\end{equation}
Here $\norm{t}_0$ denotes the {\em sparsity} of $t$, which equals the number of nonzero coordinates of $t$. The set $T_{n,s}$ is a highly symmetric, and its scaled Gaussian width satisfies
$$
\bar{w}(T_{n,s}) \asymp \sqrt{s \ln(en/s)}, 
$$
see \cite[Section~10.3.3]{V book}. Here the sign ``$\asymp$'' hides absolute constant factors. Thus, Theorem~\ref{thm: detectability radius, informal} gives
$$
r(T_{n,s}) \asymp \sqrt{s \ln(en/s)}.
$$
The proximity test in this case reduces to checking whether the sum of squares of the largest $s$ coefficients is abnormally large. The sign flipping strategy reduces to reversing the signs of the largest $s$ coefficients.


\subsection{What about non-Gaussian data?}

Real world data is rarely Gaussian. The first part of Theorem~\ref{thm: detectability radius, informal}, which identifies a regime in which the fake is detectable, can be easily generalized to non-Gaussian random vectors $X$. A version of this result, Theorem~\ref{thm: upper non-Gaussian}, holds for a completely arbitrary distribution of $X$. 

As for the second part of Theorem~\ref{thm: detectability radius, informal}, which identifies a regime in which there exist undetectable fakes, extending it to general distributions is trickier. Our proof of this result is based on sign flipping strategy, whose validity explores the symmetries of the distribution of $X$. A more general version this result, Theorem~\ref{thm: lower non-Gaussian}, holds for any random vector $X$ whose coordinates $X_i$ are independent, symmetric, and bounded random variables. It remains an interesting question to what extent we can relax the assumptions of independence, symmetry and boundedness.

\subsection{What about general sets $T$?}

While the first part of Theorem~\ref{thm: detectability radius, informal} is valid for general sets $T$, the second part requires $T$ to be highly symmetric, where the rigorous meaning of ``highly symmetric'' is given in Definition~\ref{def: highly symmetric set}. 

This latter result can fail miserably for general sets $T$: the scaled Gaussian width can hugely overestimate the detectability radius. A simple example is where $T$ consists of all unit vectors in $\R^n$ whose first coordinate is either $1/2$ or $-1/2$. An elementary calculation carried out in Section~\ref{s: non-sharpness} shows that $\bar{w}(T) \asymp \sqrt{n}$ while $r(T) \asymp O(1)$.  

In full generality, we suspect that the detectability radius $r(T)$ must be captured by some geometric quantity $\widetilde{w}(T)$ that is generally smaller than the scaled Gaussian width $\bar{w}(T)$, and which is focused only on the most important directions of $T$.  We can call such quantity the {\em focused Gaussian width}. In Section~\ref{s: focused width} we define the focused Gaussian width and show that 
$$
r(T) \lesssim \widetilde{w}(T) \le \bar{w}(T). 
$$
A fascinating problem remains whether the first inequality can be reversed for a general set $T$, i.e. whether the detectability radius is always equivalent to the focused Gaussian width.

\section{When is a fake detectable?}

Let us begin by showing that a fake is detectable whenever the radius $r$ is significantly larger than $2\bar{w}(T)$, where $\bar{w}(T)$ is the scaled Gaussian width defined in \eqref{eq: scaled Gaussian width}. The error term that quantifies ``significantly larger'' depends on the magnitude of the smallest change the adversary can make. This magnitude is the {\em inradius} of $T$, defined as
\begin{equation}	\label{eq: inradius}
	\rho(T) = \inf_{t \in T} \norm{t}_2.
\end{equation}

\begin{theorem}[When the fake is detectable]				\label{thm: upper}
	Let $X$ be a standard normal random vector in $\R^n$. Let $T \subset \R^n$ be any set and $u>0$ be any number. Assume that
	\begin{equation}	\label{eq: upper r}
		r \ge 2\bar{w}(T) + \frac{u}{\rho(T)}.
	\end{equation}
	Then there exists a set $A \subset \R^n$ satisfying 
	$$
	\Prob{X \in A} > 1-e^{-u^2/8} 
	\quad \text{and} \quad 
	\Prob{X \in A-rT} < e^{-u^2/8}.
	$$
\end{theorem}

\begin{proof}
	Consider the set
	$$
	K \coloneqq \left\{ x \in \R^n \;\vert\; \ip{x}{t} < \norm{t}_2^2 \text{ for every } t \in T \right\}.
	$$
	We claim that
	\begin{equation}		\label{eq: disjoint}
  	\frac{K}{2} \cap \Big( T - \frac{K}{2} \Big) = \emptyset.
	\end{equation}
	Indeed, if this were not true, there would exist vectors $x,y \in K$ and $t \in T$ satisfying $x/2=t-y/2$, or $x+y=2t$. Taking the scalar product with $t$ on both sides would give
	\begin{equation}		\label{eq: xyt}
	\ip{x}{t} + \ip{y}{t} = 2\ip{t}{t} = 2\norm{t}_2^2.
	\end{equation}
	On the other hand, the definition of $K$ implies $\ip{x}{t} < \norm{t}_2^2$ and $\ip{y}{t} < \norm{t}_2^2$. Adding these two inequalities creates a contradiction with \eqref{eq: xyt}, and so Claim \eqref{eq: disjoint} is proved. 
	
	We have
	\begin{align*}
		\Prob*{X \in \frac{r}{2} K}
  		&= \Prob*{\sup_{t \in T} \ip*{X}{\frac{t}{\norm{t}_2^2}} < \frac{r}{2}}
		\quad \text{(by definition of $K$)} \\
  		&\ge \Prob*{\sup_{t \in T} \ip*{X}{\frac{t}{\norm{t}_2^2}} < \bar{w}(T)+\frac{u}{2\rho(T)}}
  	\quad \text{(by assumption on $r$)} \\
  		&> 1-e^{-u^2/8},
	\end{align*}
	where the last step follows by applying the Gaussian concentration inequality (see \cite{BLM}) for the $1/\rho(T)$-Lipschitz function $f(x)=\sup_{t \in T} \ip*{x}{t/\norm{t}_2^2}$.
	
	Since \eqref{eq: disjoint} can be rewritten as
	$$
	\frac{r}{2}K \cap \Big( rT - \frac{r}{2}K \Big) = \emptyset,
	$$
	the previous bound yields
	$$
	\Prob*{X \in rT - \frac{r}{2}K}
	\le 1 - \Prob*{X \in \frac{r}{2}K}
	< e^{-u^2/8}.
	$$
	Therefore, the conclusion of the theorem holds for 
	$$
	A = \frac{r}{2}K. 
	$$
	The proof is complete.
\end{proof}

\begin{remark}[The error term is small]
    When we look at the assumption \eqref{eq: upper r}, it is helpful to regard $2\bar{w}(T)$ as the main term and $u/\rho(T)$ as the error term. The error term is typically smaller than the main term, and often much smaller. Indeed, the Gaussian width of any origin-symmetric set $T$ satisfies $ w(T) \ge \sqrt{2/\pi} \sup_{t \in T} \norm{t}_2$, see \cite[Proposition~7.5.2]{V book}. Rescaling and rearranging the terms, we conclude that the scaled Gaussian width of any origin-symmetric set $T$ satisfies
    $$
    \frac{1}{\rho(T)} \le \sqrt{\frac{\pi}{2}} \, \bar{w}(T).
    $$
    Moreover, in most interesting cases $1/\rho(T)$ is much smaller than $\bar{w}(T)$. For example, if $T$ is the unit Euclidean sphere in $\R^n$, then we have 
    $$
    \rho(T)=1 
    \quad \text{while} \quad
    \bar{w}(T) \approx \sqrt{n}.
    $$
    Similarly, if $T_{n,s}$ is the set of $s$-sparse vectors in $\R^n$ of norm at least $1$, which we introduced in Section~\ref{s: sparsity}, then we by \cite[Section~10.3.3]{V book} we have
    $$
    \rho(T)=1
    \quad \text{while} \quad
    \bar{w}(T_{n,s}) \asymp \sqrt{s \ln(en/s)}.
    $$   
\end{remark}

\subsection{How to spot a fake? A proximity test}	\label{s: proximity test}

Theorem~\ref{thm: upper} identifies a regime where a fake can be detected: there is a test that can accurately predict whether a data $x$ has been sampled from the standard normal distribution, or alternatively a point from $rT$ has been added to it. What is this test?

This test is encoded by a subset $A \subset \R^n$: upon seeing a data $x \in \R^n$, the tester returns ``real'' if $x \in A$ and ``fake'' if $x \in A^c$. In  the proof of Theorem~\ref{thm: upper}, we made the following choice of $A$:
$$
A = \frac{r}{2}K 
= \left\{ x \in \R^n \;\vert\; \ip{x}{t} < \frac{r}{2} \norm{t}_2^2 \text{ for every } t \in T \right\}.
$$
Rewriting the inequality above as $\ip{x}{rt} < \frac{1}{2} \norm{rt}_2^2$, we see that it is simply saying that the point $x$ is closer to the origin than to the point $rt$. 
Thus, the proof of Theorem~\ref{thm: upper} yields the following

\smallskip

\begin{quote}
	 {\bf Proximity test.} {\em If the data point $x$ is closer to the origin than to any point in the set $rT$, return ``true''; otherwise return ``fake''.}
\end{quote}

\smallskip

If the radius $r$ satisfies \eqref{eq: upper r}, the proximity test succeeds: the probability of a false positive error and the probability of a false negative error are both bounded by $e^{-u^2/8}$.

\section{When is a fake undetectable?}

Next, let us consider the opposite situation. We will show how the adversary can create an undetectable fake whenever the radius $r$ is significantly larger than $2\bar{w}(T)$, where $\bar{w}(T)$ is the scaled Gaussian width defined in \eqref{eq: scaled Gaussian width}. The error term that quantifies ``significantly larger'' will be exactly the same as in Theorem~\ref{thm: upper} and it will depend on the inradius $\rho(T)$ of $T$, defined in \eqref{eq: inradius}. 

We will only consider highly symmetric sets of tricks $T$ in this regime, and we will  demonstrate in Section~\ref{s: non-sharpness} how the result can fail if $T$ is not highly symmetric.  

\begin{definition}[Highly symmetric set]	\label{def: highly symmetric set}
	We say that a set $T \subset \R^n$ is {\em highly symmetric} if, whenever $T$ contains a point $x$, the set $T$ must also contain any point $y \in \R^n$ that satisfies $\supp(y) = \supp(x)$ and $\norm{y}_2 \ge \norm{x}_2$.
\end{definition}

An example of a highly symmetric set is the set $T_n$ of all vectors in $\R^n$ whose norm is bounded below by $1$, which we considered in \eqref{eq: Tn}. A more general example is the set of all $s$-sparse vectors in $\R^n$ whose Euclidean norm is at least $1$, which we considered in \eqref{eq: Tns}.

\begin{theorem}[Where the fake is undetectable]			\label{thm: lower}
	Let $X$ be a standard normal random vector in $\R^n$. Let $T \subset \R^n$ be a highly symmetric set and $u>0$ be any number. Assume that
	\begin{equation}	\label{eq: lower r}
		0 \le r \le 2\bar{w}(T) - \frac{u}{\rho(T)}.
	\end{equation}
	Then for any set $A \subset \R^n$ we have 
	$$
	\Prob{X \in A-rT}
	\ge \Prob{X \in A} - e^{-u^2/8}.
	$$
\end{theorem}

\begin{proof}[Proof of Theorem~\ref{thm: lower}]
	Definition~\ref{def: highly symmetric set} of a highly symmetric set $T$ implies that the answer to the question ``does a given point $x \in \R^n$ belong to $T$?'' depends only on the support of $x$ and whether the norm of $x$ is sufficiently large. Let $S$ denote the family of subsets of $\{1,\ldots,n\}$ that are supports of points in $T$, that is 
	$$
	S \coloneqq \left\{ \supp(x):\; x \in T \right\}.
	$$
	For every set $I \in S$, let $\nu(I)$ denote the smallest Euclidean norm of a vector in $T$ whose support equals $I$: 
	$$
	\nu(I) \coloneqq \min \left\{ \norm{x}_2:\; x \in T, \; \supp(x)=I \right\}.
	$$
	Then we can express the set $T$ as follows:
	\begin{equation}	\label{eq: T highly symmetric}
	T \coloneqq \left\{ x \in \R^n:\;  \exists I \in S \text{ such that } \supp(x)=I, \; \norm{x}_2 \ge \nu(I) \right\}.
	\end{equation}
A simple computation shows that the Gaussian width and the inradius of $T$ are 
$$
\bar{w}(T) = \E \max_{I \in S} \, \frac{\norm{X_I}_2}{\nu(I)}, \quad
\rho(T) = \min_{I \in S} \, \nu(I), 
$$
where $x_I \in \R^I$ denotes the restriction of a vector $x \in \R^n$ onto the coordinates in $I$.

Apply the Gaussian concentration inequality (see \cite{BLM}) for the $1/\rho(T)$-Lipschitz function $f(x) = \max_{I \in S} \, \norm{x_I}_2/\nu(I)$. We obtain that with probability at least $1-e^{-u^2/8}$, the following holds:
\begin{equation}		\label{eq: max I large}
	\max_{I \in S} \, \frac{\norm{X_I}_2}{\nu(I)} 
	\ge \bar{w}(T)-\frac{u}{2\rho(T)}
	\ge \frac{r}{2},
\end{equation}
where the second inequality is due to the assumption on $r$.

Whenever the event in \eqref{eq: max I large} occurs, there exists a set $I = I(X) \in S$ such that 
$$
\norm{2X_I}_2 \ge r\nu(I).
$$ 
By definition of $T$, this implies that $-2X_I \in rT$, so there exists $t=t(X) \in T$ such that $-2X_I=rt$. Whenever the event \eqref{eq: max I large} does not hold, set $I(X)=I_0$ and $t(X)=t_0$ for some arbitrary but fixed $I_0 \in S$ and $t_0 \in T$. Summarizing, we have constructed a random set $I=I(X) \in S$ and a random point $t=t(X) \in T$ such that 
\begin{equation}	\label{eq: 2XI}
	\Prob*{-2X_I=rt} \ge 1-e^{-u^2/8}.
\end{equation}
Moreover, since the quantity $\norm{X_I}_2$ in the event \eqref{eq: max I large} is determined by the {\em absolute values} of the coefficients of $X$, we can arrange that the set $I$ is determined by the absolute values of the coefficients of $X$.
 
Adding the vector $-2X_I$ to the vector $X$ is equivalent to reversing the signs of the coefficients of $X$ indexed by $I$. And since the choice of $I$ is determined by the absolute values of the coefficients of $X$ and is independent of their signs, reversing the coordinates of $X$ indexed by $I$ does not change the distribution of $X$. (This is a consequence of the symmetry of the Gaussian distribution, which we explain in detail in Lemma~\ref{lem: flipping the signs} below.) Hence $X-2X_I$ has the same distribution as $X$. Therefore, for any set $A \subset \R^n$ we have
\begin{align*}
	\Prob{X \in A}
	&= \Prob{X-2X_I \in A} \\
	&\le \Prob{X-2X_I \in A \text{ and } -2X_I=rt} + \Prob{-2X_I \ne rt} \\
	&\le \Prob{X+rt \in A} + e^{-u^2/8},
\end{align*}
where the in last step we used \eqref{eq: 2XI}. Since $t \in T$, rearranging the terms completes the proof. 
\end{proof}

Let us now explain the crucial fact used in the proof above, namely that sign flipping does not change the distribution. 

\begin{lemma}[Sign flipping]			\label{lem: flipping the signs}
	Let $X=(X_1,\ldots,X_n)$ be a random vector in $\R^n$ whose coordinates $X_i$ are independent and have symmetric distributions.\footnote{A random variable $\xi$ has symmetric distribution if $\xi$ has the same distribution as $-\xi$.} Let $\theta = (\theta_1,\ldots,\theta_n) \in \{-1,1\}^n$ be a random vector whose value is determined by the absolute values of the coefficients of $X$. Then the random vector $(\theta_1 X_1, \ldots, \theta_n X_n)$ has the same distribution as $X$. 
\end{lemma}

\begin{proof}
	The random vector of interest has coefficients
	$$
	\theta_i X_i = \theta_i \sgn(X_i) \abs{X_i}.
	$$
	Condition on the absolute values of the coefficients of $X$. This fixes the values of $\abs{X_i}$ and $\theta_i$, while by symmetry $\sgn(X_i)$ are (conditionally) independent Rademacher random variables. Since the signs $\theta_i$ are now fixed, the symmetry of Rademacher distribution implies that 
	$$
	\theta_i \sgn(X_i) \equiv \sgn(X_i),
	$$
	where we use the sign ``$\equiv$'' to indicate the equality of (conditional) joint distributions of the coefficients. Multiplying by the fixed numbers $\abs{X_i}$, we get 
	$$
	\theta_i \sgn(X_i) \abs{X_i} \equiv \sgn(X_i) \abs{X_i}. 
	$$
	In other words, 
	$$
	\theta_i X_i \equiv X_i.
	$$
	Since the conditional distributions are equal almost surely (in fact, deterministically), the original distributions are equal, too. 
\end{proof}

\subsection{Undetectability}	\label{s: undetectability}
Let us explain why the conclusion of Theorem~\ref{thm: lower} can be interpreted as existence of undetectable fakes. 

Any fake detection test can be encoded by a subset $A \subset \R^n$: upon seeing a data $x \in \R^n$, the tester returns ``real'' if $x \in A$ and ``fake'' if $x \in A^c$. 
Rewriting the conclusion of Theorem~\ref{thm: lower} as
$$
\Prob{X \in A^c} + \Prob{X \in A-rT} \ge 1-e^{-u^2/8},
$$
we see that the two probabilities on the left hand side can not be simultaneously small: we must have 
$$
\min \Big( \Prob{X \in A^c}, \, \Prob{X \in A-rT} \Big)
\ge \frac{1}{2}-\frac{1}{2}e^{-u^2/8}.
$$
This means that no test can reliably detect a fake: either the false positive rate or the false negative rate must be close to $50\%$ or higher.

\subsection{How to evade detection? A sign flipping strategy}	\label{s: sign flipping}

Theorem~\ref{thm: lower} tells us that for any radius $r$ below a certain threshold, the adversary has a strategy to create an undetectable. What is this strategy? 

To create a fake, the adversary looks at the true data $x$, picks a point $t=t(x) \in T$, and outputs $x+rt$. The proof of Theorem~\ref{thm: lower} contains a recipe for choosing $t$. We express the set $T$ via \eqref{eq: T highly symmetric}, find a set of indices
$$
I = \argmax_{I \in S} \frac{\norm{x_I}_2}{\nu(I)}
$$
and choose $t \in T$ such that $rt = -2x_I$. (And if such $t$ does not exist, the adversary gives up.) Adding such $rt$ to $x$ is equivalent to reversing the signs of the coefficients of $x$ indexed by $I$. Thus, the proof of Theorem~\ref{thm: lower} yields the following strategy for the adversary:

\smallskip

\begin{quote}
	{\bf Sign flipping strategy.} {\em Given a data point $x$ and a set $T$ as in \eqref{eq: T highly symmetric}, choose a set of indices $I$ that maximizes the ratio $\norm{x_I}_2/\nu(I)$ and reverse the signs of the coefficients of $x$ indexed by $I$. Do this only if such a reversal can be realized as adding some point from $rT$ to $x$. Otherwise give up.}
\end{quote}

\smallskip

If the radius $r$ satisfies \eqref{eq: lower r}, then the sign flipping strategy succeeds with high probability. It creates fake data that is statistically indistinguishable from the true data. As we pointed out in Section~\ref{s: undetectability}, for any test either the false positive rate or the false negative rate must be close to $50\%$ or higher.

\section{Extensions for non-Gaussian data}

Since real world data is rarely Gaussian, it is natural to wonder whether our results  generalize to random vectors $X$ with general distributions. 

Particularly straightforward is the generalization of Theorem~\ref{thm: upper}, which identifies the regime in which fakes are detectable. Consider straightforward generalizations of the concepts of Gaussian width \eqref{eq: Gaussian width} and scaled Gaussian width \eqref{eq: scaled Gaussian width}, in which the standard normal random vector $X$ is replaced with a given arbitrary random vector $X$ taking values in $\R^n$. This leads to the concept of {\em $X$-width} and {\em scaled $X$-width}, respectively:
$$
w_X(T) = \E \sup_{t \in T} \ip*{X}{t}
\quad \text{and} \quad
\bar{w}_X(T) = \E \sup_{t \in T} \ip*{X}{\frac{t}{\norm{t}_2^2}}.
$$ 
Then Theorem~\ref{thm: upper} generalizes as follows.

\begin{theorem}[Where the fake is detectable: general distributions]				\label{thm: upper non-Gaussian}
	Let $X$ be any random vector taking values in $\R^n$. 
	Let $T \subset \R^n$ be any origin-symmetric set\footnote{The assumption that $T$ is origin symmetric, i.e. that $T=-T$, is a bit annoying, but it is unavoidable in this statement, since a set $T$ consisting of a single point must satisfy $\bar{w}(T)=0$. However, it is easy to get around this assumption by replacing any set $T$ with the origin-symmetric set $T \cup -T$.} and $u>0$ be any number. Assume that
	$$
	r > 2 u \cdot \bar{w}_X(T).
	$$
	Then there exists a set $A \subset \R^n$ such that 
	$$
	\Prob{X \in A} > 1-1/u 
	\quad \text{and} \quad 
	\Prob{g \in A-rT} < 1/u.
	$$
\end{theorem}

\begin{proof}
	Follow the proof of Theorem~\ref{thm: upper} but use Markov's inequality instead of the Gaussian concentration inequality. We assumed that $T$ is origin-symmetric to make sure that the random variable $\sup_{t \in T} \ip*{X}{t/\norm{t}_2^2}$ takes only non-negative values, which makes Markov's inequality applicable. 
\end{proof}

As for Theorem~\ref{thm: upper}, which identifies the regime in which fakes are detectable, it is less clear to which non-Gaussian distributions it can be generalized. The proof of Theorem~\ref{thm: upper} uses a sign-flipping strategy that relies crucially on certain symmetries of the Gaussian distribution. It can be extended to any distribution that is sufficiently symmetric:

\begin{theorem}[Where the fake is undetectable: non-Gaussian data]			\label{thm: lower non-Gaussian}
	Let $X=(X_1,\ldots,X_n)$ be a random vector whose coordinates $X_i$ are independent, symmetric random variables taking values in the interval $[-1,1]$. Let $T \subset \R^n$ be a highly symmetric set and $u>0$ be any number. Assume that
	$$
	r \le 2\bar{w}(T) - \frac{u}{\rho(T)}.
	$$
	Then for any set $A \subset \R^n$ we have 
	$$
	\Prob{X \in A-rT}
	\ge \Prob{X \in A} - 2e^{-u^2/16}.
	$$
\end{theorem}

\begin{proof}
	Follow the proof of Theorem~\ref{thm: lower} but use Talagrand's convex concentration inequality \cite[Theorem~6.6]{Talagrand} instead of the Gaussian concentration inequality.
\end{proof}

\section{Is the detectability radius equivalent to the focused width?}	\label{s: focused width}

\subsection{The upper bound is not always sharp}		\label{s: non-sharpness}

Theorem~\ref{thm: upper} says that a fake can be detected if the radius $r$ is somewhat larger than $2\bar{w}(T)$. Theorem~\ref{thm: lower} demonstrates that this result is optimal for for highly symmetric sets $T$. 

If a set $T$ does not have enough symmetries, Theorem~\ref{thm: lower} can fail. Consider, for example, the set of all unit vectors whose first coordinate equals $\pm 1/2$, that is
\begin{equation}	\label{eq: T non-sharpness}
	T = \left\{ t \in \R^n:\; \norm{t}_2 = 1, \; \abs{t_1} = \frac{1}{2} \right\}.
\end{equation}
A simple calculation yields
$$
\bar{w}(T) \asymp \sqrt{n},
$$ 
where the ``$\asymp$'' sign hides absolute constant factors. Thus, Theorem~\ref{thm: upper} says that fake detection is possible whenever the radius satisfies $r \gtrsim \sqrt{n}$.

However, this result is too conservative. The detection in this example is possible even if the radius is an {\em absolute constant}. For example, if $r=100$, the first coordinate of the fake vector $X+rt$ is 
$$
(X+rt)_1 = X_1 \pm 50.
$$
Comparing this to the first coordinate of the real vector $X$, which is $X_1 \sim N(0,1)$, we see that the first coordinate of any fake vector must have a huge non-zero mean. This can be easily tested. Hence the scaled Gaussian width vastly overestimates the detection radius: 
$$
r(T) = O(1)
\quad \text{while} \quad
\bar{w}(T) \asymp \sqrt{n},
$$
This violates Theorem~\ref{thm: lower}.

\subsection{Focused width}

The example above makes us wonder what causes the scaled Gaussian to overestimate the detection radius, and how to fix this problem. The definition of the Gaussian width takes into account all directions $t \in T$, while not all directions are equally useful to an adversary: some tricks $t \in T$ may be more ``revealing'' and thus less useful. We wonder if forcing the Gaussian width to {\em focus} on the most important directions might solve the attention problem. 

This motivates the following definition of the {\em focused Gaussian width}. To make it more broadly applicable, let us define it not just for Gaussian $X$ but for general distributions. 

\begin{definition}[Focused Gaussian width]
	The {\em focused Gaussian width} of $T$ is defined as 
	$$
	\widetilde{w}(T) = \inf_S w(S)
	$$
	where the infimum is over all origin-symmetric sets $S \subset \R^n$ satisfying
	\begin{equation}	\label{eq: focused width S}
		\forall t \in T \; \exists s \in S :\; \ip{t}{s} \ge 1.
	\end{equation}
	In other words, $\widetilde{w}(T)$ is the smallest Gaussian width of an origin-symmetric set whose polar is disjoint from $T$. 
\end{definition}

To make it more broadly applicable, if $X$ is an arbitrary random vector taking values in $\R^n$, the {\em focused $X$-width} is defined as
$$
\widetilde{w}_X(T) = \inf_S w_X(S).
$$

The set $S$ in this definition encodes the directions of ``focus'', and taking the infimum over $S$ encourages the width to focus on the most important directions. 

To compare with the scaled width, it is convenient to rewrite the definition of the focused $X$-width equivalently as follows: 
\begin{equation}	\label{eq: focused width alternative}
	\widetilde{w}_X(T) = \inf_H \bar{w}_X(H)
\end{equation}
where the infimum is over all measurable sets $H \subset \R^n$ satisfying
\begin{equation}	\label{eq: polar condition}
	\forall t \in T \; \exists h \in H:\; \ip{t}{h} \ge \norm{h}_2^2.
\end{equation}
To check the equivalence, choose $h=s/\norm{s}_2^2$, or $s=h/\norm{h}_2^2$.

Choosing $H=T$ and $h=t$ in \eqref{eq: polar condition}, we immediately obtain:

\begin{lemma}[Focusing can only reduce the width]
	For any random vector $X$ taking values in $\R^n$ and for any set $T \subset \R^n$, we have
	$$
	\widetilde{w}_X(T) \le \bar{w}_X(T).
	$$
\end{lemma}

\subsection{Does focus help?}

The next result states that Theorem~\ref{thm: upper} can be strengthened by replacing the width with the focused width.

\begin{theorem}[Where the fake is detectable: focused width]	\label{thm: upper focused}
	Let $X$ be any random vector taking values in $\R^n$. 
	Let $T \subset \R^n$ be any origin-symmetric set and $u>0$ be any number. Let
	$$
	r > 2 u \cdot \widetilde{w}_X(T).
	$$
	Then there exists a set $A \subset \R^n$ such that 
	$$
	\Prob{X \in A} > 1-1/u 
	\quad \text{and} \quad 
	\Prob{X \in A-rT} < 1/u.
	$$
\end{theorem}

\begin{proof}
	By the assumption on $r$ and the alternative definition of the focused width \eqref{eq: focused width alternative}, there exists a set $H \subset \R^n$ satisfying \eqref{eq: polar condition} and such that 
	$$
	r > 2u \cdot \bar{w}_X(H).
	$$

	Consider the set
	$$
	K \coloneqq \left\{ x \in \R^n \;\vert\; \ip{x}{h} < \norm{h}_2^2 \text{ for every } h \in H \right\}.
	$$
	We claim that
	\begin{equation}		\label{eq: disjoint H}
  	\frac{K}{2} \cap \Big( T - \frac{K}{2} \Big) = \emptyset.
	\end{equation}
	Indeed, if this were not true, there would exist vectors $x,y \in K$ and $t \in T$ such that $x/2=t-y/2$, or $x+y=2t$. By \eqref{eq: polar condition}, we can find a vector $h \in H$ satisfying 
	$$
	\ip{t}{h} \ge \norm{h}_2^2.
	$$
	Taking the scalar product with $h$ on both sides of the identity $x+y=2t$ would give
	\begin{equation}		\label{eq: xyh}
	\ip{x}{h} + \ip{y}{h} = 2\ip{t}{h} \ge 2\norm{h}_2^2.
	\end{equation}
	On the other hand, by definition of $K$ we have $\ip{x}{h} < \norm{h}_2^2$ and $\ip{y}{h} < \norm{h}_2^2$. This contradicts \eqref{eq: xyh}, and so Claim \eqref{eq: disjoint H} is proved. 
	
	Then	
	\begin{align*}
		\Prob*{X \in \frac{r}{2} K}
  		&= \Prob*{\sup_{h \in H} \ip*{X}{\frac{h}{\norm{h}_2^2}} < \frac{r}{2}}
		\quad \text{(by definition of $K$)} \\
  		&\ge \Prob*{\sup_{h \in H} \ip*{X}{\frac{h}{\norm{h}_2^2}} < u \cdot \bar{w}_X(H)}
  	\quad \text{(by assumption on $r$)} \\
  		&> 1-1/u,
	\end{align*}
	where the last step follows by applying Markov's inequality. Note that the assumption that $T$ is origin-symmetric guarantees that the random variable $\sup_{t \in T} \ip*{X}{t/\norm{t}_2^2}$ takes only non-negative values, which makes Markov's inequality applicable. 
	
	Since \eqref{eq: disjoint H} can be rewritten as
	$$
	\frac{r}{2}K \cap \Big( rT - \frac{r}{2}K \Big) = \emptyset,
	$$
	the previous bound yields
	$$
	\Prob*{X \in rT - \frac{r}{2}K}
	\le 1 - \Prob*{X \in \frac{r}{2}K}
	< 1/u.
	$$
	Therefore, the conclusion of the theorem holds for 
	$$
	A = \frac{r}{2}K. 
	$$
	The proof is complete.
\end{proof}

\subsection{Revisiting a challenging example} 
Let us revisit the example from Section~\ref{s: non-sharpness}, which shows that the usual Gaussian width can vastly overestimate the detection radius. There we considered the set
$$
T = \left\{ x \in \R^n:\; \norm{x}_2 = 1, \; \abs{x_1} = \frac{1}{2} \right\}.
$$
 	We noticed a discrepancy: the scaled Gaussian width of $T$ is of the order of $\sqrt{n}$, while the fake detection is possible even when the radius $r$ is around an absolute constant.
	
	Let us now compute the {\em focused} Gaussian width of $T$. Consider the set 
	$$
	S = \left\{ -2e_1, 2e_1 \right\}
	$$
	where $e_1 = (1,0,0,\ldots,0)$. 	
	The condition \eqref{eq: focused width S} is satisfied, and so 
	$$
	\widetilde{w}(T) \le w(S) = 2 \E \abs{g} = 2 \sqrt{\frac{2}{\pi}}.
	$$
	So the focused Gaussian width of $T$ is $O(1)$ -- much smaller than the usual Gaussian width. 
	
	As opposed to the Gaussian width, the focused Gaussian width correctly estimates the detectability radius in this example. Indeed, Theorem~\ref{thm: upper focused} guarantees that the fake detection is possible if the radius is of the order of $O(1)$. In fact, we described such a test in Section~\ref{s: non-sharpness}: check if the magnitude of the first coordinate is abnormally large.

\subsection{A conjecture}

Inspired by the example above, one can wonder if the focused Gaussian width is always equivalent to the detectability radius, i.e. 
$$
r(T) \asymp \widetilde{w}(T).
$$
Theorem~\ref{thm: upper focused} gives $r(T) \lesssim \widetilde{w}(T)$. It remains a question whether this inequality can always be reversed. In other words, does a version of Theorem~\ref{thm: lower} hold for any origin-symmetric set $T$ if we replace the scaled Gaussian width $\bar{w}(T)$ with the focused Gaussian width $\widetilde{w}(T)$? 

Finally,
we conjecture that if detection is feasible for a particular value $r_{1}$, then it remains feasible for all values 
$ r>r_{1}$; however, we currently lack a formal proof of this claim.

\subsection{Acknowledgment} The authors sincerely thank the anonymous referees for their valuable comments, which helped improve this paper.

\end{document}